\documentclass[11pt,reqno,a4paper]{amsart}

\usepackage{graphicx}
\usepackage{pifont}
\usepackage{asymptote}
\usepackage{asypictureB}
\usepackage{enumitem}
\usepackage{caption}
\usepackage{setspace}
\setlist[enumerate]{font={\upshape}, label=\arabic*., leftmargin=2.5em}
\setlist[itemize]{leftmargin=2.5em}
\setlist[description]{leftmargin=\parindent, 
	itemsep=3pt
}

\newlist{equivlist}{enumerate}{1}
\setlist[equivlist]{font={\upshape}, label=(\roman*)}
\usepackage{tikz}
\usetikzlibrary{matrix,intersections,calc}
\tikzset{ 
	table/.style={
		matrix of nodes,
		nodes={rectangle,text width=1.75em,align=center},
		text depth=1.25ex,
		text height=2.5ex,
		nodes in empty cells
	}
}
\newcommand\doverline[1]{%
	\tikz[baseline=(nodeAnchor.base)]{
		\node[inner sep=0] (nodeAnchor) {$#1$}; 
		\draw[line width=0.1ex,line cap=round] 
		($(nodeAnchor.north west)+(0.0em,0.2ex)$) 
		--
		($(nodeAnchor.north east)+(0.0em,0.2ex)$) 
		($(nodeAnchor.north west)+(0.0em,0.5ex)$) 
		--
		($(nodeAnchor.north east)+(0.0em,0.5ex)$) 
		;
}}
\usepackage{mathtools}
\usepackage{amsfonts}
\usepackage{amssymb}
\usepackage{cite} 
\usepackage[sort,compress]{natbib}
\usepackage{hypernat} 

\usepackage[pdftex
,pagebackref
]{hyperref}
\hypersetup{
	colorlinks=true,
	allcolors=blue
}
\usepackage[capitalize]{cleveref}
\usepackage[margin=0.7in]{geometry}
\newtheorem{theorem}{Theorem}[section]
\newtheorem{lemma}[theorem]{Lemma}
\newtheorem{claim}[theorem]{Claim}
\newtheorem*{claimu}{Claim}
\Crefname{claim}{Claim}{Claims}
\newlist{lemlist}{enumerate}{1}
\setlist[lemlist]{font={\upshape}, label={\upshape(\alph*)},ref={\thelemma(\alph*)},leftmargin=*}

\newtheorem{conjecture}[theorem]{Conjecture}
\Crefname{conjecture}{Conjecture}{Conjectures}

\expandafter\let\expandafter\oldproof\csname\string\proof\endcsname
\let\oldendproof\endproof
\renewenvironment{proof}[1][\proofname]{%
	\oldproof[\normalfont\bfseries #1]%
}{\oldendproof}
\newenvironment{subproof}[1][\normalfont\it Proof]{%
	\begin{proof}[#1]%
	}{%
	\end{proof}%
}

\newcommand{\dd}{\textquotedblleft}
\newcommand{\ee}{\textquotedblright}

\newcommand{\mac}{\mathcal}

\newcommand{\eps}{\varepsilon}

\newcommand{\nin}{\notin}

\renewcommand{\subset}{\subseteq}

\newcommand{\erh}{Erd\H{o}s-Hajnal}

\DeclarePairedDelimiter\abs{\lvert}{\rvert}%
\DeclarePairedDelimiter\ceil{\lceil}{\rceil}%
\DeclarePairedDelimiter\floor{\lfloor}{\rfloor}%
\makeatletter
\newcommand{\leqnomode}{\tagsleft@true}
\newcommand{\reqnomode}{\tagsleft@false}
\makeatother

\date{November 13, 2023; revised \today}
\begin{document}
	\title{Induced subgraph density. VI. Bounded VC-dimension}
	\author{Tung Nguyen}
	\address{Princeton University, Princeton, NJ 08544, USA}
	\email{\href{mailto:tunghn@math.princeton.edu}
		{tunghn@math.princeton.edu}}
	\author{Alex Scott}
	\address{Mathematical Institute,
		University of Oxford,
		Oxford OX2 6GG, UK}
	\email{\href{mailto:scott@maths.ox.ac.uk}{scott@maths.ox.ac.uk}}
	\author{Paul Seymour}
	\address{Princeton University, Princeton, NJ 08544, USA}
	\email{\href{mailto:pds@math.princeton.edu}{pds@math.princeton.edu}}
	\thanks{The first and the third authors were partially supported by AFOSR grant FA9550-22-1-0234 and NSF grant DMS-2154169.
		The second author was supported by EPSRC grant EP/X013642/1.}
	\begin{abstract}
		We confirm a conjecture of Fox, Pach, and Suk, that for every $d>0$, there exists $c>0$ such that every $n$-vertex graph of VC-dimension at most $d$ has a clique or stable set of size at least~$n^c$.
		This implies that, in the language of model theory, every graph definable in NIP structures has a clique or anti-clique
  of polynomial size, settling a conjecture of Chernikov, Starchenko,~and Thomas.
  
		Our result also implies that every two-colourable tournament satisfies the tournament version of the \erh{} conjecture, which completes the verification of the conjecture for six-vertex tournaments.
		The result extends to uniform hypergraphs of bounded VC-dimension as well.
		
		The proof method uses the ultra-strong regularity lemma for graphs of bounded VC-dimension proved by 
		Lov\'asz and Szegedy and the method of iterative sparsification introduced by the authors in an earlier paper.
	\end{abstract}
	\maketitle
	\section{Introduction}
	All (hyper)graphs in this paper are finite and simple.
	A graph $H$ is an {\em induced subgraph} of a graph $G$ if $H$ can be obtained from $G$ by removing vertices.
	A class $\mac C$ of graphs is {\em hereditary} if it is closed under taking induced subgraphs and under isomorphism;
	and a hereditary class $\mac C$ is {\em proper} if it is not the class of all graphs.
	We say that $\mac C$ has the {\em \erh{} property} 
	if there exists $c>0$ such that every graph $G\in \mac C$ has a clique or stable set of size at least $|G|^c$, where $|G|$
	denotes the number of vertices of $G$.
	A conjecture of Erd\H os and Hajnal~\cite{MR1031262,MR599767} (see~\cite{MR1425208,MR3150572} for surveys and~\cite{MR4563865,density1,density4,density7} for some recent partial results) asserts that:
	\begin{conjecture}
		\label{conj:eh}
		Every proper hereditary class of graphs has the \erh{} property.
	\end{conjecture}
	For a set $\mac F$ of subsets of a set $V$, a subset $S$ of $V$ is {\em shattered} by $\mac F$ if for every $A\subset S$ there exists $B\in\mac F$ with $B\cap S=A$.
	The {\em VC-dimension} of $\mac F$ (introduced by Vapnik and Chervonenkis in~\cite{MR3408730}) is the largest cardinality of a subset of $V$ that is shattered by $\mac F$.
	Since its introduction in 1971, the notion of VC-dimension has proved to be relevant in a number of areas of pure and applied mathematics.
	The {\em VC-dimension} of a graph $G$ is the VC-dimension of the set $\{N_G(v):v\in V(G)\}$ of subsets of $V(G)$, where $N_G(v)$ denotes the set of all neighbours of $v$ (not including $v$ itself).
	It is not hard to see that for every $d\ge1$, the class of graphs of VC-dimension at most $d$ is a proper hereditary class.
	The aim of this paper is to confirm a conjecture of Fox, Pach, and Suk~\cite{MR3943496} that every hereditary class of graphs of bounded VC-dimension has the \erh{} property;
	in their paper, they came close to settling this by proving a bound of $2^{(\log n)^{1-o(1)}}$ where the constant depending on the VC-dimension is hidden in the $o(1)$ term. (In this paper $\log$ denotes the binary logarithm.) Our result is:
	\begin{theorem}
		\label{thm:vc}
		For every $d\ge1$, there exists $c_d>0$ such that every graph $G$ of $VC$-dimension at most $d$ contains a clique or stable set of size at least $|G|^{c_d}$.
	\end{theorem}
	
	(A careful inspection of our proof yields $c_d\ge c_{d-1}^{K}$ for every $d\ge2$ where $K$ is a universal constant; and so one can take $c_d\ge 2^{-2^{O(d)}}$. We omit the details.)
	
	As a result, this paper can be viewed as part of a growing body of 
	results~\cite{MR4622593,MR3943496,janzer2021,MR4357431} that, when VC-dimension is bounded,  completely settle or significantly improve bounds for well-known open problems in extremal combinatorics.
	
	The story behind \cref{thm:vc}, perhaps, began in geometric graph theory with the result of Larman, Matou\v sek, Pach, and T\"or\H ocsik~\cite{MR1272297} that the class of intersection graphs of line segments in general position in the plane has the \erh{} property (in fact, they proved more, for intersection graphs of convex compact sets in the plane).
	Alon, Pach, Pinchasi, Radoi\v ci\'c, and Sharir~\cite{MR2156215} generalized this result to classes of semi-algebraic graphs of bounded description complexity.
	Fox, Pach, and T\'oth~\cite{MR2776643} provided another extension of~\cite{MR1272297} by verifying the \erh{} property of the class of string graphs where every two curves cross a bounded number of times
	(and a recent result by Tomon~\cite{tomon2020} shows that the condition \dd every two curves cross a bounded number of times\ee{} can be dropped).
	In yet another direction, Sudakov and Tomon~\cite{MR4502021} recently proved \cref{conj:eh} for the classes of algebraic graphs of bounded description complexity, which is an analogue of the result of~\cite{MR2156215}.
	All of these hereditary classes (except intersection graphs of convex compact sets, and more generally, string graphs) turn out to have bounded VC-dimension.
	Indeed,
	\begin{itemize}
		\item for the classes 
		of semi-algebraic graphs of bounded description complexity, this is true by the classical Milnor--Thom theorem 
		in real algebraic geometry (see~\cite{MR1899299});
		
		\item 
		for the classes of algebraic graphs of bounded description complexity,
		this is a consequence of a theorem of R\'onyai, Babai, and Ganapathy~\cite{MR1824986} on the number of zero-patterns of polynomials; and
		
		\item for the classes of string graphs where every two curves intersect at a bounded number of points, this follows from a 
		result of Pach and T\'oth~\cite[Lemma 4.2]{MR2279670} together with the standard fact that a hereditary class has bounded VC-dimension 
		if and only if it does not contain a bipartite graph, the complement of a bipartite graph, and a split graph (see \cref{lem:vcind}).
	\end{itemize}
	
	In a model-theoretic setting, \cref{thm:vc} states that every class of graphs definable in NIP (non-independence property) structures has the \erh{} property (see~\cite{MR3560428} for a general reference on NIP theories),
	which was formally stated as a conjecture by Chernikov, Starchenko, and Thomas~\cite{MR4205776}. 
	Two notable special cases of NIP graphs include distal graphs and stable graphs.\footnote{This is 
	different from the notion of stable vertex subsets in graphs;
	in graph-theoretic language, it means graphs not containing a fixed half-graph bipartite pattern. See also~\cite{MR3789671}.}
	Malliaris and Shelah~\cite{MR3145742,MR4386783} implicitly proved \cref{conj:eh} for stable graphs (which contains the result of~\cite{MR4502021}) by developing regularity lemmas for these graphs
	(see~\cite{MR3731711} for a short proof using pseudo-finite model theory).
	In the case of distal graphs, Basu~\cite{MR2595744} proved the \erh{} property for graphs definable by o-minimal structures (which in fact extends~\cite{MR2156215}),
	before Chernikov and Starchenko~\cite{MR3852184} made use of the theory of Keisler measures in NIP to formulate regularity lemmas for distal graphs and settle the general problem in this direction (see also~\cite{MR3503725} for a short and pure model-theoretic proof).
	Recently, Fu~\cite{fu2023} also combined tools from model theory and a result from~\cite{MR4563865} to prove the \erh{} property of the class of graphs of VC-dimension at most two.
	
	In the combinatorial setting, \cref{thm:vc} turns out to be equivalent to the fact that every tournament not containing a fixed two-colourable subtournament has a transitive vertex subset of polynomial size;
	we will discuss this in more detail in \cref{subsec:tour}.
	\cref{thm:vc} also extends to the setting of uniform hypergraphs, allowing us to give tighter asymptotics on Ramsey numbers for $k$-uniform hypergraphs with bounded VC-dimension; we will discuss this in \cref{subsec:hyper}.

	%
	
	As in previous papers of the series, we need to work in a more general context, and will prove a stronger result:
	our main result says that graphs of bounded VC-dimension satisfy the polynomial form of a theorem of R\"odl.
	For a graph $G$, let $\overline G$ denote its complement; and for $\eps\in(0,\frac12)$, we say that $G$ is {\em $\eps$-sparse} if $G$ has maximum degree at most $\eps\abs G$, and {\em $\eps$-restricted} if one of $G,\overline G$ is $\eps$-sparse.  We recall that
	R\"odl's theorem~\cite{MR837962} then states the following:
	\begin{theorem}
		\label{thm:rodl}
		For every proper hereditary class $\mac C$ of graphs and every $\eps\in(0,\frac12)$, there exists $\delta>0$ such that every graph $G\in\mac C$ contains an $\eps$-restricted induced subgraph on at least $\delta\abs G$ vertices.
	\end{theorem}

The dependence of $\delta$ on $\eps$ here is an important question.  R\"odl's proof used Szemer\'edi's regularity lemma, and gave tower-type bounds.  A better dependence was obtained by
Fox and Sudakov~\cite{MR2455625}, who gave a proof with $\delta =2^{-c\log^2(1/\eps)}$;
this was improved to $\delta =2^{-c\log^2(1/\eps)/\log\log(1/\eps)}$ in \cite{density1}.
Fox and Sudakov~\cite{MR2455625} conjectured that in fact the dependence can be taken to be polynomial, which would imply the Erd\H os-Hajnal conjecture.

	Let us say a hereditary class $\mac C$ has the {\em polynomial R\"odl property} if there exists $C>0$ such that for 
	every $\eps\in(0,\frac12)$, every graph $G\in\mac C$ contains an $\eps$-restricted induced subgraph on at least $\eps^C\abs G$ vertices.
	The Fox--Sudakov conjecture is then the following strengthening of \cref{conj:eh}:
	\begin{conjecture}
		\label{conj:fs}
		Every proper hereditary class of graphs has the polynomial R\"odl property. 
	\end{conjecture}
	Our main result says that every hereditary class of graphs of bounded VC-dimension does indeed satisfy \cref{conj:fs}.
	\begin{theorem}
		\label{thm:vcpolyrodl}
		For every $d\ge1$, the class of graphs of VC-dimension at most $d$ has the polynomial R\"odl property.
	\end{theorem}

Our proof of \cref{thm:vcpolyrodl} uses the ultra-strong regularity lemma for graphs of bounded VC-dimension proved
by Lov\'asz and Szegedy, and builds on the method of iterative sparsification introduced in earlier papers of the series \cite{density3,density4}.  The method involves passing through a sequence of induced subgraphs that are successively more restricted (see \cref{subsec:sparsify}), and naturally leads to the extra generality of \cref{thm:vcpolyrodl}.

We would like to remark that Buci\'c, Fox, and Pham~\cite{bfp} recently proved that the \erh{} property is in fact equivalent to the polynomial R\"odl property for any hereditary class of graphs, which implies that \cref{thm:vc,thm:vcpolyrodl} are equivalent.
Nevertheless, for our proof method means that it is more convenient to prove the result in the stronger polynomial R\"odl form; and in the case of graphs of bounded VC-dimension, such an equivalence can also be obtained via an application of the ultra-strong regularity lemma (see \cref{thm:cond} and the argument at the end of \cref{subsec:sparsify}).
	\section{Ultra-strong regularity, bigraphs and nearly pure blockades}\label{sec:reg}

A natural approach to proving Erd\H os-Hajnal-type results is to attempt to find induced subgraphs that are highly structured; the challenge is to control the size and structure of these subgraphs.  In this section, we work with large subgraphs that are `approximately' nicely structured: up to some noise, they look like a blowup of some much smaller graph.  These will be helpful in the next section, where we look for subgraphs where the structure is more tightly controlled.

A {\em blockade} in a graph $G$ is a sequence $\mac B=(B_1,\ldots,B_k)$ of pairwise disjoint (and possibly empty) subsets of $V(G)$, called {\em blocks}; the {\em length} of $\mac B$ is $k$ and the {\em width} is $\min_{i\in[k]}|B_i|$. For $\ell,w\ge0$, we say that $\mac B$ is an {\em $(\ell,w)$-blockade} if it has length at least $\ell$ and width at least $w$.

 We will be interested in blockades where the pairs $(B_i,B_j)$ satisfy certain density conditions.  Two disjoint subsets $A,B$ of vertices are {\em anticomplete} in $G$ if $G$ has no edge between $A,B$;
	{\em complete} in $G$ if $A,B$ are anticomplete in $\overline G$;
	and {\em pure} if $A,B$ are either complete or anticomplete in $G$.
A blockade $\mac B=(B_1,\ldots,B_k)$ is {\em anticomplete} in $G$ if $B_i,B_j$ are anticomplete in $G$ for all distinct $i,j\in[k]$,
	{\em complete} in $G$ if $\mac B$ is anticomplete in $\overline G$
 and {\em pure} if, for all distinct $i,j\in[k]$, $B_i,B_j$ are either complete or anticomplete.  Note that for a blockade, being pure is a far weaker condition than being a complete or anticomplete blockade.

 In general, complete or anticomplete blockades will be very desirable; but we will first need to work with blockades that are only pure, and that only satisfy approximate versions of the conditions. 
 For $\eps>0$ and a graph $G$, two disjoint subsets $A,B$ are {\em weakly $\eps$-sparse} in $G$
	if $G$ has at most $\eps\abs A\abs B$ edges between $A,B$,
	and {\em weakly $\eps$-pure} if they are weakly $\eps$-sparse in $G$ or $\overline G$.
	We say that $B$ is {\em $\eps$-sparse} to $A$ in $G$ if every vertex of $B$ is adjacent in $G$ to at most $\eps\abs A$ vertices in $A$; the pair $(A,B)$ is {\em $\eps$-pure} if each is $\eps$-sparse to the other in $G$ or $\overline G$. 
 A blockade
 $\mac B$ is   
{\em $\eps$-pure} 
 or {\em weakly $\eps$-pure} if every pair $B_i,B_j$ satisfies the corresponding condition.

The first step in the proof of \cref{thm:main} is to find a large $\eps$-pure blockade.
For this we use the ultra-strong regularity lemma of Lov\'asz and Szegedy~\cite[Section 4]{MR2815610} 
	which was reproved by Fox, 
	Pach, and Suk~\cite[Theorem 1.3]{MR3943496} via a short proof that gives better bounds 
	(a weaker version for bipartite graphs was also developed by Alon, Fisher, and Newman~\cite{MR2341924}). The ultra-strong regularity lemma says:
	\begin{theorem}
		\label{thm:partition}
		For every $d\ge1$, there exists $K\ge2$ such that the following holds.
		For every $\eps\in(0,\frac12)$ and every graph $G$ with VC-dimension at most $d$,
		there exist $L\in[\eps^{-1},\eps^{-K}]$ and an equipartition $V(G)=V_1\cup\cdots\cup V_L$ such that all but at most an $\eps$ fraction of the pairs $(V_i,V_j)$ are weakly $\eps$-pure in $G$.
	\end{theorem}
	Here, an {\em equipartition} of a set $S$ is a partition $S_1\cup\cdots\cup S_k$ of $S$ such that $\abs{S_i},\abs{S_j}$ differ by at most one for all $i,j\in[k]=:\{1,\ldots,k\}$.

It will be helpful to work with bipartite patterns in $G$.  Let us make this more formal. A {\em bigraph} is a graph $H$
	together with a bipartition $(V_1(H), V_2(H))$ of $V(H)$.
If $H$ is a bigraph and $v\in V(H)$, we denote by $H\setminus v$ the bigraph $H'$ obtained by deleting $v$, where for $i = 1,2$,
	$V_i(H')=V_i(H)\setminus \{v\}$ if $v\in V_i(H)$, and $V_i(H')=V_i(H)$ otherwise. 
	For a bigraph $H$, a {\em bi-induced} copy of $H$ in a graph $G$ is an injective map 
	$\varphi\colon V(H)\to V(G)$ such that, for all $u\in V_1(H)$ and $v\in V_2(H)$, $uv\in E(H)$ if and only if $\varphi (u)\varphi(v)\in E(G)$.
	We recall the following standard fact (see~\cite[Lemma 3.3]{MR3416129}).
	\begin{lemma}
		\label{lem:vcbip}
		For every bigraph $H$, there exists $d\ge1$ such that every graph of VC-dimension at least $d$ contains a bi-induced copy of $H$.
		Conversely, for every $d\ge1$, there is a bigraph $H$ such that every graph containing a bi-induced copy of $H$ has VC-dimension at least $d$.
	\end{lemma}

	In view of this, \cref{thm:partition} can be restated as follows.
	\begin{theorem}
		\label{thm:ultra}
		For every bigraph $H$, there exists $K\ge2$ such that for every $\eps\in(0,\frac12)$ and every graph $G$ with no 
		bi-induced copy of $H$, there exists $L\in[\eps^{-1},\eps^{-K}]$ for which there is an equipartition $V(G)=V_1\cup\cdots\cup V_L$ 
		such that all but at most an $\eps$ fraction of the pairs $(V_i,V_j)$ are weakly $\eps$-pure.
	\end{theorem}

We use \cref{thm:partition} to obtain a blockade that is $\eps$-pure and large (more specifically, we want the length and width to have polynomial dependence on $\eps$).
 
	\begin{theorem}
		\label{thm:cond}
		For every bigraph $H$, there exists $b\ge2$ such that for every $\eps\in(0,\frac12)$ and every graph $G$ with $\abs G\ge \eps^{-b}$ and no bi-induced copy of $H$, there is an $(\eps^{-1},\eps^b\abs G)$-blockade $(B_1,\ldots,B_{\ell})$ in $G$ such that
		\begin{itemize}
			\item $\abs{B_1}=\cdots=\abs{B_{\ell}}\le \eps^2\abs G$; and
			
			\item for all distinct $i,j\in[\ell]$, $B_i,B_j$ are $\eps$-sparse to each other in $G$ or $\overline G$.
		\end{itemize}
	\end{theorem}
	\begin{proof}
		Let $K\ge2$ be given by \cref{thm:ultra};
		we claim that $b:=5K$ satisfies the theorem.
		To this end, by \cref{thm:ultra} with $\eps^4$ in place of $\eps$,
		if $G$ is a graph with no bi-induced copy of $H$, then there is an equipartition $V(G)=V_1\cup\cdots\cup V_L$ with 
		$L\in[\eps^{-4},\eps^{-4K}]$ such that all but at most an $\eps^4$ fraction of the pairs $(V_i,V_j)$ are weakly $\eps^4$-pure.
		By Tur\'an's theorem, there exists $J\subseteq[L]$ with $\abs J\ge\frac12\eps^{-4}$ such that $(V_i,V_j)$ is weakly $\eps^4$-pure for all distinct $i,j\in J$.
		Then there exists $I\subseteq J$ with $\abs I\ge \frac12\abs J\ge\frac14\eps^{-4}\ge\eps^{-1}$ such that $\abs{V_i}=\abs{V_j}$ for all distinct $i,j\in I$.
		Let $\ell:=\ceil{\eps^{-1}}$; it follows that there exists $I\subseteq J$ with $I=[\ell]$ such that $\abs{V_i}=\abs{V_j}$ for all distinct $i,j\in I$.
		For every $i\in I$, let $B_i$ be the set of vertices $v$ in $V_i$ such that for every $j\in I\setminus\{i\}$,
		\begin{itemize}
			\item $v$ has at most $\frac12\eps\abs{V_j}$ neighbours in $V_j$ if $(V_i,V_j)$ is weakly $\eps^4$-sparse in $G$; and
			
			\item $v$ has at most $\frac12\eps\abs{V_j}$ nonneighbours in $V_j$ if $(V_i,V_j)$ is weakly $\eps^4$-sparse in $\overline G$.
		\end{itemize}
		Then
		\[\abs{B_i}\ge \abs{V_i}-(\ell-1)\cdot 2\eps^3\abs{V_i}
		\ge\abs{V_i}-\eps^{-1}\cdot2\eps^3\abs{V_i}
		=(1-2\eps^2)\abs{V_i}
		\ge\abs{V_i}/2\]
		and by removing vertices if necessary we may assume that $\abs{B_i}=\ceil{\abs{V_i}/2}=\ceil{m/2}$ where $m=\abs{V_i}$.
		It follows that for all distinct $i,j\in I$, $B_i,B_j$ are $\eps$-sparse to each other in $G$ or $\overline G$.
		Also, since 
		$$m\ge \floor{\abs G/L}\ge\abs G/(2L)\ge \eps^{4K+1}\abs G$$ 
		(as $\abs G\ge \eps^{-b}=\eps^{-5K}$),
		it follows that for each $i\in I$,
		$$\abs {B_i}\ge m/2\ge \eps^{4K+2}\abs G\ge \eps^{5K}\abs G=\eps^b\abs G$$
		and 
		$$\abs{B_i}\le m\le \ceil{\abs G/L}\le 2\abs G/L\le 2\eps^4\abs G\le\eps^2\abs G.$$
		This proves \cref{thm:cond}.  
	\end{proof}

	\section{Iterative sparsification}\label{subsec:sparsify}

In this section we give the proof of \cref{thm:vcpolyrodl}.
We recall that a graph is {\em $\eps$-restricted} if either the graph or its complement is $\eps$-sparse.
 In view of \cref{lem:vcbip}, \cref{thm:vcpolyrodl} 
 is equivalent to:
	\begin{theorem}
		\label{thm:main}
		For every bigraph $H$, there exists $C>0$ such that for every $\eps\in(0,\frac12)$, every graph $G$ with no bi-induced copy of $H$ contains an $\eps$-restricted induced subgraph with at least $\eps^C\abs G$ vertices.
	\end{theorem}

The proof of \cref{thm:main} uses the framework of {\em iterative sparsification}, which was introduced in earlier papers from this series~\cite{density3,density4} (and will be used also in \cite{density7}).   
The goal is to find an $\eps$-restricted induced subgraph of size at least $\operatorname{poly}(\eps)|G|$.  Rather than doing this in one step, we instead attempt to move through a sequence of induced subgraphs that are successively sparser (in $G$ or its complement): given a $y$-restricted induced subgraph $F$, we search for an induced subgraph that is $\operatorname{poly}(y)$-restricted and is at most a $\operatorname{poly}(y)$ factor smaller.  Provided we can start the process, and do not get stuck on the way, the following lemma shows that this gives the required subset. 

	\begin{lemma}
		\label{lem:sparsify}
		Let $c\in(0,1)$, $a\ge2$, and $t\ge1$.
		Suppose that $x\in(0,c)$ and $G$ is a graph satisfying:
		\begin{itemize}
			\item there is a $c$-restricted induced subgraph of $G$ with at least $c^t\abs G$ vertices; and
			
			\item for every $y\in[x,c]$ and every $y$-restricted induced subgraph $F$ of $G$ with $\abs F\ge y^{2t}\abs G$,
			there is a $y^a$-restricted induced subgraph of $F$ with at least $y^{at}\abs F$ vertices.
		\end{itemize}
		Then $G$ contains an $x$-restricted induced subgraph with at least $x^{2at}\abs G$ vertices.
	\end{lemma}
	\begin{proof}
		By the first condition of the lemma, there exists $y\in[x^a,c]$ minimal such that $G$ has a $y$-restricted induced subgraph $F$ with $\abs F\ge y^{2t}\abs G$.
		If $y\ge x$, then by the second condition of the lemma and since $a\ge2$, $F$ has a $y^a$-restricted induced subgraph with at least $y^{at}\abs F\ge y^{at+2t}\abs G\ge y^{2at}\abs G$ vertices;
		but this contradicts the minimality of $y$ since $x^a\le y^a<y$.
		Thus $x^a\le y<x$; and so $F$ is $x$-restricted, and $\abs F\ge y^{2t}\abs G\ge x^{2at}\abs G$.
		This proves \cref{lem:sparsify}.
	\end{proof}

We will find a subgraph satisfying the first bullet by using R\"odl's theorem (\cref{thm:rodl}), with a suitable $t=t(c)$.  However, finding a subgraph that satisfies the second bullet is more challenging, and we need to allow for an alternative `good' outcome.  We will show in \cref{lem:polyrodl} that if we get stuck then we can instead
find a large complete or anticomplete blockade (note that this is much stronger than being pure; and the blockades given by \cref{thm:cond} are only $\eps$-pure).
Let us show that, if we can find sufficiently large blockades, then we can obtain the 
\erh{} property; we will return to the (stronger) polynomial R\"odl property at the end of the section.
	\begin{lemma}
		\label{lem:eh}
		Let $\mac C$ be a hereditary class of graphs.
		Suppose that there exists $d\ge2$ such that for every $x\in(0,2^{-d})$ and every $G\in\mac C$,
		either:
		\begin{itemize}
			\item $G$ has an $x$-restricted induced subgraph with at least $x^d\abs G$ vertices; or
			
			\item there is a complete or anticomplete $(k,\abs G/k^d)$-blockade in $G$, for some $k\in[2,1/x]$.
		\end{itemize}
		Then there exists $a\ge2$ such that every $n$-vertex graph in $\mac C$ has a clique or stable set of size at least~$n^{1/a}$.
	\end{lemma}
	\begin{proof}
		A {\em cograph} is a graph with no induced four-vertex path;
		and it is well-known that every $k$-vertex cograph has a clique or stable set of size at least $k^{1/2}$.
		Thus, it suffices to prove by induction that every $G\in\mac C$ contains an induced cograph of size at least $\abs G^{1/(2d^2)}$.
		We may assume $\abs G> 2^{2d^2}$.
		Let $x:=\abs G^{-1/(2d)}\in(0,2^{-d})$.
		By the hypothesis, either:
		\begin{itemize}
			\item there exists $S\subseteq V(G)$ with $\abs S\ge x^d\abs G$ such that $G[S]$ is $x$-restricted; or
			
			\item there is a complete or anticomplete $(k,\abs G/k^d)$-blockade in $G$, for some $k\in[2,1/x]$.
		\end{itemize}
		
		If the first bullet holds, then since $\abs S\ge x^d\abs G\ge x^{-1}$,
		Tur\'an's theorem gives a clique or stable set in $G[S]$ of size at least
		$(2x)^{-1}>x^{-1/2}=\abs G^{1/(4d)}$.
		If the second bullet holds, then since $k\ge2$, by applying induction to each block in the blockade to obtain a cograph and combining them, we see that $G$ contains an induced cograph of size at least
		$$k(\abs G/k^d)^{1/(2d^2)}=k^{1-1/(2d)}\abs G^{1/(2d^2)}>\abs G^{1/(2d^2)}.$$
		This proves \cref{lem:eh}.
	\end{proof}

 The key step in making the iterative sparsification strategy work is therefore to show that if the second bullet of \cref{lem:sparsify} does not hold then we can find a sufficiently large complete or anticomplete blockade. We will show this in 
\cref{lem:vc3}: given a $y$-restricted graph $F$ with no bi-induced $H$, we will prove that we can either pass to the desired $\operatorname{poly}(y)$-restricted subgraph or find a complete or anticomplete blockade whose length and width depend polynomially on $y$.  We will argue by induction on $|H|$, and grow the blockade one block at a time: the first step (\cref{lem:vc1}) is to find a complete or anticomplete pair $(A,B)$, where $A$ has size $\operatorname{poly}(y)|F|$ and $B$ contains all but a small fraction of the rest of $F$; we then (\cref{lem:vc2}) repeat the argument inside $B$, continuing until we obtain a blockade.  Restricted graphs can either be dense or sparse: we will assume for the moment that our restricted graph is sparse, and handle the dense case later by taking complements.
 
	

	\begin{lemma}
		\label{lem:vc1}
		Let $H$ be a bigraph, and let $v\in V(H)$.
		Let $b\ge2$ be given by \cref{thm:cond}.
		Assume there exists $a\ge2$ such that
		every $n$-vertex graph with no bi-induced copy of $H\setminus v$ contains a clique or stable set of size at least $n^{1/a}$.
		Let $y\in(0,1/\abs H)$, and let $F$ be a $y$-sparse graph with no bi-induced copy of $H$.
		Then either:
		\begin{itemize}
			\item $F$ has a $y^{2a}$-restricted induced subgraph with at least $y^{3ba^2}\abs F$ vertices; or
			
			\item there is an anticomplete pair $(A,B)$ in $F$ with $\abs A\ge y^{3ba^2}\abs F$ and $\abs B\ge (1-3y)\abs F$.
		\end{itemize}
	\end{lemma}
	\begin{proof}  We have a sparse graph, and want to find an anticomplete pair $(A,B)$.  We will do this by first using \cref{thm:cond} (a corollary of the ultra-strong regularity lemma) to find a large, $\operatorname{poly}(y)$-pure blockade and then looking at how the rest of the graph attaches to it. 
 So let $\eps:=y^{3a^2}$, and suppose that the first outcome does not hold;
		then $\abs F>y^{-3ba^2}=\eps^{-b}$.  By
		\cref{thm:cond}, $F$ has a $(\eps^{-1},\eps^b\abs F)$-blockade $(B_1,\ldots,B_{\ell})$ with $\ell=\ceil{\eps^{-1}}$, such that:
		\begin{itemize}
			\item $\abs{B_1}=\cdots=\abs{B_\ell}\le2\eps^2\abs F$; and
			
			\item for all distinct $i,j\in[\ell]$,
			$B_i,B_j$ are $\eps$-sparse to each other in $F$ or $\overline F$.
		\end{itemize}
		
		Let $D:=V(F)\setminus(B_1\cup\cdots\cup B_{\ell})$ and $m:=\abs{B_1}$.
		For $i\in[\ell]$, a vertex $v\in D$ is {\em mixed} on $B_i$ if it has a neighbour and a nonneighbour in $B_i$.
		\begin{claimu}
			Every vertex in $D$ is mixed on at most $y\ell$ of the blocks $B_1,\ldots,B_{\ell}$.
		\end{claimu}
		\begin{subproof}
			Suppose there is a vertex $w\in D$ mixed on at least $y\ell$ of the blocks $B_1,\ldots,B_\ell$,
			say $B_1,\ldots,B_r$ where $r\ge y\ell\ge y\eps^{-1}=y^{1-3a^2}$.
			Let $J$ be the graph with vertex set $[r]$ where for all distinct $i,j\in[r]$,
			$ij\nin E(J)$ if and only if $B_i,B_j$ are $\eps$-sparse to each other in $F$.
			
			We claim that there is no bi-induced copy of $H\setminus v$ in $J$.
			Suppose that there is; and we may assume $V(H\setminus v)\subseteq V(J)$.
			We assume that $v\in V_1(H)$ without loss of generality.
			For each $u\in V_2(H)$, let $w_u$ be a neighbour of $w$ in $B_u$ if $uv\in E(H)$ and a nonneighbour of $w$ in 
			$B_u$ if $uv\nin E(H)$.
			For each $z\in V_1(H)\setminus\{v\}$ and $u\in V_2(H)$, since $uz\nin E(H)$ if and only if $uz\nin E(J)$ if and only if $B_u,B_z$ are $\eps$-sparse to each other in $F$,
			$w_u$ is adjacent in $F$ to at most $\eps\abs{B_z}$ vertices in $B_z$ if $uz\nin E(H)$ and nonadjacent in $G$ to at most $\eps\abs{B_z}$ vertices in $B_z$ if $uz\in E(H)$.
			Thus, for each $z\in V_1(H)\setminus\{v\}$, there are at least (note that $\eps\le y<1/\abs H$)
			\[\abs{B_z}-\abs {V_2(H)}\eps\abs {B_z}
			\ge \abs {B_z}-\abs H y\abs {B_z}>0\]
			vertices $z'\in B_z$ such that for every $u\in Y$, $w_uz'\in E(F)$ if and only if $uz\in E(H)$;
			let $w_z$ be such a vertex.
			It follows that $\{w\}\cup\{w_z:z\in V_1(H)\setminus\{v\}\}\cup\{w_u:u\in Y\}$ forms a bi-induced copy of $H$ in $F$, a contradiction.
			
			Thus, there is no bi-induced copy of $H\setminus v$ in $J$.
			By the choice of $a$, $J$ thus contains a clique or stable set $I$ with $\abs I\ge\abs J^{1/a}\ge (y^{1-3a^2})^{1/a}=y^{-3a+1/a}$.
			Let $S:=\bigcup_{i\in I}B_i$; then $\abs S=\abs Im$ since $\abs{B_i}=m$ for all $i\in I$.
			If $I$ is a stable set in $J$, then $F[S]$ has maximum degree at most
			\[m+\abs I\eps m=(1/\abs I+\eps)\abs Im
			\le (y^{3a-1/a}+y^{3a})\abs S\le 2y^{3a-1}\abs S\le y^{2a}\abs S;\]
			and similarly, if $I$ is a clique in $J$, then $\overline F[S]$ has maximum degree at most $y^{2a}\abs S$.
			Thus $F[S]$ is $y^{2a}$-restricted which is the first outcome of the lemma, a contradiction.
			This proves the claim.
		\end{subproof}
		
		Now, by the claim, there exists $i\in[\ell]$ such that there are at most $y\abs D$ vertices in $D$ that are mixed on $B_i$.
		Since $F$ is $y$-sparse, there are at most $y\abs F$ vertices in $D$ that are complete to $B_i$.
		Thus, since 
		\[\abs{B_1}+\cdots+\abs{B_\ell}\le \ell\eps^2\abs F\le 2\eps\abs F=2y^{3a^2}\abs F\le y\abs F,\]
		there are at least
		\[\abs F-y\abs D-y\abs F-(\abs{B_1}+\cdots+\abs{B_\ell})
		\ge (1-3y)\abs F\]
		vertices in $F$ with no neighbour in $B_i$.
		Because $\abs{B_i}\ge \eps^b\abs F=y^{3ba^2}\abs F$, the second outcome of the lemma holds.
		This proves \cref{lem:vc1}.
	\end{proof}
	We now apply \cref{lem:vc1} repeatedly to move from an anticomplete pair to an anticomplete blockade.
	\begin{lemma}
		\label{lem:vc2}
		Let $H$ be a bigraph, and let $v\in V(H)$.
		Let $b\ge2$ be given by \cref{thm:cond}.
		Assume there exists $a\ge2$ such that
		every $n$-vertex graph with no bi-induced copy of $H\setminus v$ contains a clique or stable set of size at least $n^{1/a}$.
		Let $0<y\le 2^{-12\abs H}$, and let $F$ be a $y$-sparse graph with no bi-induced copy of $H$.
		Then either:
		\begin{itemize}
			\item $F$ has a $y^a$-restricted induced subgraph with at least $y^{2ba^2}\abs F$ vertices; or
			
			\item there is an anticomplete $(y^{-1/2},y^{2ba^2}\abs F)$-blockade in $F$.
		\end{itemize}
	\end{lemma}
	\begin{proof}
		Suppose that the second outcome does not hold.
		Let $n\ge0$ be maximal such that there is a blockade $(B_0,B_1,\ldots,B_n)$ in $F$ with $\abs{B_n}\ge(1-3y^{1/2})^n\abs F$ and $\abs{B_{i-1}}\ge y^{2ba^2}\abs F$ for all $i\in[n]$.
		Since the second outcome does not hold, $n<y^{-1/2}$; and so, since $y\le2^{-12}$,
		\[\abs{B_n}\ge(1-3y^{1/2})^n\abs F\ge 4^{-3y^{1/2}n}\abs F
		\ge 4^{-3}\abs F\ge y^{1/2}\abs F.\]
		Hence $F[B_n]$ has maximum degree at most $y\abs F\le y^{1/2}\abs{B_n}$; and so \cref{lem:vc1} (with $y^{1/2}$ in place of $y$, note that $y^{1/2}\le 2^{-6\abs H}<1/\abs H$) implies that either:
		\begin{itemize}
			\item there exists $S\subseteq B_n$ with $\abs S\ge y^{3ba^2/2}\abs{B_n}$ such that $F[S]$ is $y^a$-restricted; or
			
			\item there is an anticomplete pair $(A,B)$ in $F[B_n]$ with $\abs A\ge y^{3ba^2/2}\abs {B_n}$ and $\abs B\ge (1-3y^{1/2})\abs{B_n}$.
		\end{itemize}
		If the second bullet holds, then $(B_0,B_1,\ldots,B_{n-1},A,B)$ would be a blockade contradicting the maximality of $n$ since $\abs A\ge y^{3ba^2/2}\abs{B_n}\ge y^{3ba^2/2+1/2}\abs F\ge y^{2ba^2}\abs F$.
		Thus the first bullet holds; and so $S\subseteq V(F)$ is $y^{a}$-restricted in $F$ and satisfies $\abs S\ge y^{3ba^2/2}\abs{B_n}\ge y^{3ba^2/2+1/2}\abs G\ge y^{2ba^2}\abs F$.
		Hence the first outcome of the lemma holds.
		This proves \cref{lem:vc2}.
	\end{proof}

 We have been assuming that our restricted graphs are sparse.  We now move to the general case, by handling dense and sparse graphs simultaneously.
 
 For a bigraph $H$, its {\em bicomplement} is the bigraph $\doverline H$ with the same bipartition and edge set 
	$\{uv:u\in V_1(H),v\in V_2(H),uv\nin E(H)\}$.
 
	\begin{lemma}
		\label{lem:vc3}
		Let $H$ be a bigraph, and let $v\in V(H)$.
		Let $b$ be given by \cref{thm:cond}.
		Assume there exists $a\ge2$ such that every $n$-vertex graph with no bi-induced copy of $H\setminus v$ contains a clique or stable set of size at least $n^{1/a}$.
		Let $0<y\le 2^{-12\abs H}$, and let $F$ be a $y$-restricted graph with no bi-induced copy of $H$.
		Then either:
		\begin{itemize}
			\item $F$ has a $y^{a}$-restricted induced subgraph with at least $y^{2ba^2}\abs F$ vertices; or
			
			\item there is a complete or anticomplete $(y^{-1/2},y^{2ba^2}\abs F)$-blockade in $F$.
		\end{itemize}
	\end{lemma}
	\begin{proof}
            Since $F$ is $y$-restricted, either $F$ or $\overline F$ is $y$-sparse.  If $F$ is $y$-sparse then the result follows from \cref{lem:vc2} (using the assumption on $H\setminus v$).  If $\overline F$ is $y$-sparse then the result follows from applying \cref{lem:vc2} to the complement of $F$ (noting that $\overline F$ is $\doverline H$-free, and using the assumption on $\doverline H\setminus v$).
	\end{proof}

With \cref{lem:vc3} in hand, we can now prove that the conditions of \cref{lem:eh} are satisfied.

	\begin{lemma}
		\label{lem:polyrodl}
		For every bigraph $H$, there exists $d\ge2$ such that for every $x\in(0,2^{-d})$ and every graph $G$ with no bi-induced copy of $H$,
		either:
		\begin{itemize}
			\item $G$ has an $x$-restricted induced subgraph with at least $x^d\abs G$ vertices; or
			
			\item there is a complete or anticomplete $(k,\abs G/k^d)$-blockade in $G$, for some $k\in[2,1/x]$.
		\end{itemize}
	\end{lemma}
	\begin{proof} 
		We argue by induction on $\abs H$. We may assume that $|H|\ge 2$. Choose $v\in V(H)$.
		By \cref{lem:eh} and the induction hypothesis applied to $H\setminus v$ and $\doverline H\setminus v$,
		there exists $a\ge4$ such that 
		every $n$-vertex graph with no bi-induced copy of $H\setminus v$, and every $n$-vertex graph with no bi-induced copy of 
		$\doverline H\setminus v$, contains a clique or stable set of size at least $n^{1/a}$.  
		Let $c:=2^{-12\abs H}$, and let $b\ge2$ be given by \cref{thm:cond}.  By R\"odl's Theorem~\ref{thm:rodl}, we can choose some
		$t\ge ba^2$ such that every graph $G$ with no bi-induced copy of $H$ contains a $c$-restricted induced subgraph with at least $c^t\abs G$ vertices.
		We claim that $d:=2\max(at,\abs H)\ge 8t$ satisfies the lemma.

  To show this, let $x\in(0,2^{-d})\subseteq(0,c)$, and suppose that $G$ has no bi-induced copy of $H$.
		Suppose that the second outcome of the lemma does not hold; that is, there is no $k\in[2,1/x]$ such that there is a complete or anticomplete $(k,\abs G/k^d)$-blockade in $G$.
		\begin{claimu}
			For every $y\in[x,c]$ and every $y$-restricted induced subgraph $F$ or $G$ with $\abs F\ge y^{2t}\abs G$, there is a $y^a$-restricted induced subgraph of $F$ with at least $y^{at}\abs F$ vertices.
		\end{claimu}
		\begin{subproof}
			By the choice of $t$ and \cref{lem:vc3}, either:
			\begin{itemize}
				\item $F$ has a $y^a$-restricted induced subgraph with at least $y^{2ba^2}\abs F\ge y^{2t}\abs F\ge y^{at}\abs F$ vertices; or
				
				\item there is a complete or anticomplete $(y^{-1/2},y^{2ba^2}\abs F)$-blockade in $F$.
			\end{itemize}
			
			If the second bullet holds, then since $y^{2ba^2}\abs F\ge y^{4t}\abs G\ge y^{d/2}\abs G$ by the choice of $d$,
			there would be a complete or anticomplete $(y^{-1/2},y^{d/2}\abs G)$ blockade in $G$, which contradicts that the second outcome of the lemma does not hold (note that $y^{1/2}\le c^{1/2}\le\frac12$).
			Thus the first bullet holds, proving the claim.
		\end{subproof}
		
		\cref{lem:sparsify} and the claim imply that $G$ has an $x$-restricted induced subgraph with at least $x^{2at}\abs G\ge x^d\abs G$ vertices,
		which is the first outcome of the lemma.
		This proves \cref{lem:polyrodl}.
	\end{proof}
	\cref{lem:eh,lem:polyrodl} show the \erh{} property for every class of graphs with no bi-induced 
	copy of a given bipartite graph. But much of the argument so far was about density, and that allows us to go a step further,  and apply \cref{thm:cond} once more to prove \cref{thm:main}.
	\begin{proof}
		[Proof of \cref{thm:main}]
		It follows from \cref{lem:eh,lem:polyrodl} that there exists $a\ge2$ such that every $n$-vertex graph with no bi-induced copy of $H$ has a clique or stable set of size at least $n^{1/a}$;
		and by increasing $a$ if necessary we may assume $2^a>\abs H$.
		Let $b\ge2$ be given by \cref{thm:cond}.
		We claim that $C:=2ab$ satisfies the theorem.
		To see this, let $\eps\in(0,\frac12)$ and $G$ be a graph with no bi-induced copy of $H$.
		It suffices to show that $G$ has an $\eps$-restricted induced subgraph with at least $\eps^{2ab}\abs G$ vertices.
  
		We may assume that $\abs G>\eps^{-2ab}$.
		By \cref{thm:cond} with $\eps^{2a}$ in place of $\eps$, there is an $(\eps^{-2a},\eps^{2ab}\abs G)$-blockade $(B_1,\ldots,B_{\ell})$ in $G$ where $\ell\ge\eps^{-2a}$, such that
		$\abs{B_1}=\cdots=\abs{B_{\ell}}=:m$; and for all distinct $i,j\in[\ell]$, $B_i,B_j$ are $\eps^{2a}$-sparse to each other in $G$ or $\overline G$.
		Let $J$ be the graph with vertex set $[\ell]$ where $ij\nin E(J)$ if and only if $B_i,B_j$ are $\eps^{2a}$-sparse to each other in $G$.
		\begin{claimu}
			$J$ has no bi-induced copy of $H$.
		\end{claimu}
		\begin{subproof}
			Suppose not; and we may assume $V(H)\subseteq V(J)$.
			For every $i\in V(H)$, let $v_i$ be a uniformly random vertex in $B_i$, chosen independently.
			Then for $(X,Y)$ the bipartition of $H$, the probability that $\{v_1,\ldots,v_{\abs H}\}$ forms a bi-induced copy of $H$ in $G$ is at least $1-\abs X\abs Y\eps^{2a}>1-\abs H^22^{-2a}>0$ by the choice of $a$;
			and so there is a bi-induced copy of $H$ in $G$, a contradiction.
			This proves the claim.
		\end{subproof}
		Now, by the claim and the choice of $a$, $J$ has a clique or stable set $I$ with $\abs I\ge\abs J^{1/a}\ge \eps^{-2}$. 
		Let $S:=\bigcup_{i\in I}B_i$; then $\abs S=\abs Im\ge m\ge \eps^{2ab}\abs G$.
		If $I$ is stable in $J$, then $G[S]$ has maximum degree at most 
		\[m+\abs I\eps^am
		=(1/\abs I+\eps^a)\abs Im
		\le 2\eps^2\abs S\le \eps\abs S. \]
		Similarly, if $I$ is a clique in $J$, then $\overline G[S]$ has maximum degree at most $\eps\abs S$.
		This proves \cref{thm:main}.
	\end{proof}
	\section{Ramsey numbers for hypergraphs with bounded VC-dimension}
	\label{subsec:hyper}
	
	A {\em  $k$-uniform hypergraph} is a hypergraph whose edges all have cardinality $k$.   The notion of VC-dimension extends naturally to $k$-uniform hypergraphs $\mac H$.
For a set $S$, let $\binom Sk$ denote the family of subsets of size $k$ of $S$; 
 and for $S\in\binom{V(\mac H)}{k-1}$, let $N_{\mac H}(S):=\{v\in V(\mac H):S\cup\{v\}\in E(\mac H)\}$.
	The {\em VC-dimension} of $\mac H$ is the VC-dimension of the family $\{N_{\mac H}(S):S\in\binom{V(\mac H)}{k-1}\}$ with ground set $V(\mac H)$.

	A {\em clique}  in a $k$-uniform hypergraph $\mac H$
	 is a subset of $V(\mac H)$ whose subsets of size $k$ are 
	all in $E(\mac H)$, and a {\em stable set} is a clique in the complementary $k$-uniform hypergraph.  
The {\em Ramsey number} $R_k(t)$ is the smallest integer $n\ge1$ such that every $n$-vertex $k$-graph contains a clique or stable set of size at least~$n$.  By Ramsey's theorem, $R_k(n)$ exists for all values of $n$ and $k$; equivalently, writing $f_k(n)$ for the largest $m$ such  that every $k$-uniform hypergraph with $n$ vertices has a clique or stable set of size at least~$m$, we have $f_k(n)\to\infty$ for any fixed $k$.  The growth rate of $f_k$ has received considerable attention: for $k=2$, it is well known that $f_k(n)=\Theta(\log n)$, but for $k\ge3$ the bounds are much farther apart.  For $k=3$, Erd\H os, Hajnal and Rado \cite{EHR} showed that
$$c_1\log\log n\le f_3(n)\le c_2\sqrt{\log n};$$
and for $k\ge 4$, the best bounds are
$$c_1\log^{(k-1)}(n)\le f_k(n)\le c_2\sqrt{\log^{(k-2)}n},$$
where $\log^{(k)}n$ denotes the $k$-times iterated logarithm. 

What can be said about hypergraphs with bounded VC-dimension?  Let $f_k^d(n)$ be the largest integer $m$ such that every $k$-uniform hypergraph with $n$ vertices and VC-dimension at most $d$ has a clique or stable set of size at least~$m$.  For $k\ge 3$, the best previous bounds are
$$e^{(\log^{(k-1)}n)^{1-o(1)}}\le f_k^d(n)\le c_2\log^{(k-2)}(n),$$
where the 
upper follows from a construction of Conlon, Fox, Pach, Sudakov, and Suk~\cite{MR3217709} and the
lower bound is due to Fox, Pach, and Suk~\cite{MR3943496}, who observed that it follows from their $2^{(\log n)^{1-o(1)}}$ bound for graphs of bounded VC-dimension and by an adaptation of the classical Erd\H os--Rado greedy argument~\cite{MR0065615}.

Stronger lower bounds have been obtained in more restrictive settings.  Conlon, Fox, Pach, Sudakov, and Suk~\cite{MR3217709} proved that every {semi-algebraic} $n$-vertex $k$-uniform hypergraph of bounded description complexity admits a clique or stable set of size at least $\exp(c\log^{(k-1)}n)$, where the constant $c>0$ depends on $k$ and the description complexity of the hypergraph. This was later extended to $k$-uniform hypergraphs definable by {distal structures} by Chernikov, Starchenko, and Thomas~\cite{MR4205776}.  Combining the Erd\H os–Rado argument with our \cref{thm:vc} allows us to extend this further to hypergraphs with bounded VC-dimension, improving the lower bound to show that, for fixed $k\ge 3$ and $d\ge 1$,
$$f_k^d(n)=e^{\Theta(\log^{(k-1)}n)}.$$

We give the details for completeness.  For integers $d\ge1$, $k\ge2$, and $n\ge1$, let $R_k^d(n)$ be the smallest integer $m\ge1$ such that every $m$-vertex $k$-graph of VC-dimension at most $d$ contains a clique or stable set of size at least~$n$;
	then \cref{thm:vc} says that $R_2^d(n)\le n^C$ for all $n\ge1$, for some $C>0$ depending on $d$ only.
	\begin{theorem}
		\label{thm:hyper}
		For all integers $d\ge1$, $k\ge3$, and $n\ge2$,
		\[R_k^d(n)\le 2^{\binom{R_{k-1}^d(n-1)}{k-1}}+k-2.\]
	\end{theorem}
	\begin{proof}
		Let $p:=R_{k-1}^d(n-1)$.
		Let $\mac H$ be a $k$-uniform hypergraph with $2^{p\choose k-1}+k-2$ vertices and VC-dimension at most $d$.
		We claim that:
		\begin{claim}
			\label{claim:hyper}
			For every integer $q$ with $k-2\le q\le p$, there are disjoint $A_q,B_q\subseteq V(\mac H)$ with $\abs{A_q}=q$ and $\abs{B_q}\ge 2^{{p\choose k-1}-{q\choose k-1}}$ such that for every $S\subseteq{A_q\choose k-1}$, either $S\cup\{v\}\in E(\mac H)$ for all $v\in B_q$ or $S\cup\{v\}\nin E(\mac H)$ for all $v\in B_q$.
		\end{claim}
		\begin{subproof}
			The claim is vacuously true for $q=k-2$.
			For $k-2\le q<p$, we shall prove the claim of $q+1$ assuming that it is true for $q$.
			Let $u\in B_q$ be arbitrary, and let $A_{q+1}:=A_q\cup\{u\}$.
			For every $T\in{A_q\choose k-2}$ and $v\in B_q\setminus\{u\}$,
			let $f_T(v)$ be $0$ if $T\cup\{u,v\}\in E(\mac H)$ and $1$ if $T\cup\{u,v\}\nin E(\mac H)$.
			Then by the pigeonhole principle,
			there exists $B_{q+1}\subseteq B_q$ with
			\begin{align*}
				\abs{B_{q+1}}\ge\ceil*{2^{-{\abs{A_q}\choose k-2}}(\abs{B_q}-1)}
				&\ge \ceil*{2^{-{q\choose k-2}}\left(2^{{p\choose k-1}-{q\choose k-1}}-1\right)}\\
				&=\ceil*{2^{{p\choose k-1}-{q+1\choose k-1}}-2^{-{q\choose k-2}}}
				=2^{{p\choose k-1}-{q+1\choose k-1}},
			\end{align*}
			such that $f_T(v)=f_T(w)$ for all $T\in{A_q\choose k-2}$ and $v,w\in B_{q+1}$;
			here the last equation holds since ${q+1\choose k-1}\le{p\choose k-1}$ and ${q\choose k-2}\ge1$.
			Thus, by the choice of $A_q,B_q$, it follows that for every $S\subseteq{A_{q+1}\choose k-1}$, either $S\cup\{v\}\in E(\mac H)$ for all $v\in B_{q+1}$ or $S\cup\{v\}\nin E(\mac H)$ for all $v\in B_{q+1}$.
			This proves \cref{claim:hyper}.
		\end{subproof}
		Now, \cref{claim:hyper} with $q=p$ gives disjoint $A_p,B_p\subseteq V(\mac H)$ with $\abs{A_p}=p$ and $\abs{B_p}\ge1$ such that for every $S\subseteq{A_p\choose k-1}$, either $S\cup\{v\}\in E(\mac H)$ for all $v\in B_p$ or $S\cup\{v\}\nin E(\mac H)$ for all $v\in B_p$.
		Let $v\in B_p$; and let $\mac H'$ be the $(k-1)$-uniform hypergraph with vertex set $A_p$ where for every $S\in{A_p\choose k-1}$, $S\in E(\mac H')$ if and only if $S\cup\{v\}\in E(\mac H)$.
		\begin{claim}
			\label{claim:hypervc}
			$\mac H'$ has VC-dimension at most $d$.
		\end{claim}
		\begin{subproof}
			Suppose not; then there exists $A\subseteq A_p$ with $\abs A>d$ which is shattered by the family $\{N_{\mac H'}(T):T\in{A_p\choose k-2}\}$ with ground set $A_p$.
			However, by the definition of $\mac H'$, $A$ is then shattered by the family $\{N_{\mac H}(T\cup\{v\}):T\in{A_p\choose k-2}\}\subseteq\{N_{\mac H}(S):S\in{V(\mac H)\choose k-1}\}$ with ground set $V(\mac H)$, contrary to $\mac H$ having VC-dimension at most $d$.
			This proves \cref{claim:hypervc}. 
		\end{subproof}
		Now, by the definition of $p$, \cref{claim:hypervc} implies that $\mac H'$ has a clique or stable set $S$ with $\abs S\ge n-1$.
		Then $S\cup\{v\}$ is a clique or stable set in $\mac H$ of size at least $n$.
		This proves \cref{thm:hyper}.
	\end{proof}
	It is now not hard to iterate \cref{thm:hyper} and apply the bound $R_2^d(n)\le n^C$ to get $R_k^d(n)\le\operatorname{twr}_{k-1}(n^K)$ for some $K>0$ depending on $d,k$,
	where $\operatorname{twr}_k$ is defined recursively for $k\ge1$ by $\operatorname{twr}_1(t):=t$ and $\operatorname{twr}_{k}(t):=2^{\operatorname{twr}_{k-1}(t)}$ for all $t>0$.
	In other words, $f_k^d(n)\ge (\log^{(k-2)} n)^{1/K}$.
	\section{The viral property}
	Let us interpret \cref{thm:main} in the language of forbidden induced subgraphs.
	A {\em split} graph is a graph whose vertex set can be partitioned into a clique and a stable set.
	For graphs $H,G$, a {\em copy} of $H$ in $G$ is an injective map $\varphi\colon V(H)\to V(G)$ such that for all distinct $u,v\in V(H)$, $uv\in E(H)$ if and only if $\varphi(u)\varphi(v)\in E(G)$.
	For a finite family $\mac F$ of graphs, a graph $G$ is {\em $\mac F$-free} if there is no copy of $H$ in $G$ for all $H\in\mac F$.
	The following is a well-known strengthening of \cref{lem:vcbip} (see~\cite[Theorem 3.3]{MR3416129}).
	\begin{lemma}
		\label{lem:vcind}
		For every two bipartite graphs $H_1,H_2$ and every split graph $J$, there exists $d\ge1$ such that every $\{H_1,\overline{H_2},J\}$-free graph has VC-dimension at most $d$.
		Conversely, for every $d\ge1$, there are bipartite graphs $H_1,H_2$ and a split graph $J$ such that every graph of VC-dimension at most $d$ is $\{H_1,\overline{H_2},J\}$-free.
	\end{lemma}
	Thus, \cref{thm:vcpolyrodl} can be rewritten as:
	\begin{theorem}
		\label{thm:polyrodlvc}
		For every two bipartite graphs $H_1,H_2$ and every split graph $J$, the class of $\{H_1,\overline{H_2},J\}$-free graphs has the polynomial R\"odl property.
	\end{theorem}
	We say that a finite family $\mac F$ of graphs is {\em viral} if there exists $C>0$ such that 
	for every $\eps\in(0,\frac12)$ and for every graph $G$ with at most $(\eps^C\abs G)^{\abs H}$ copies of $H$ for all $H\in\mac F$, there is an $\eps$-restricted induced subgraph of $G$ with at least $\eps^C\abs G$ vertices.
	Thus, if $\mac F$ is viral then the class of $\mac F$-free graphs has the polynomial R\"odl property.
	It is conjectured in~\cite{density3} that $\{H\}$ is viral for all graphs $H$, which would mean that Nikiforov's theorem~\cite{MR2271833} holds with polynomial dependence.
	Recently, Gishboliner and Shapira~\cite{MR4657286} showed that for every two bipartite graphs $H_1,H_2$ and every split graph $J$,
	$\{H_1,\overline{H_2},J\}$ is viral if and only if the class of $\{H_1,\overline{H_2},J\}$-free graphs has the \erh{} property, by using a strengthening of \cref{thm:partition} for graphs with few copies of $H_1,\overline{H_2},J$.
	Later, Buci\'c, Fox, and Pham~\cite{bfp} showed that for any finite family $\mac F$, the class of $\mac F$-free graphs has the \erh{} property if and only if $\mac F$ is viral.
	Thus one can bootstrap \cref{thm:polyrodlvc} into the~following.
	\begin{theorem}
		\label{thm:vcviral}
		For every two bipartite graphs $H_1,H_2$ and every split graph $J$,  $\{H_1,\overline{H_2},J\}$ is viral.
	\end{theorem}
	\section{Tournaments}
	\label{subsec:tour}
	If $Q$ is a tournament, a tournament is {\em $Q$-free} if it has no subtournament isomorphic to $Q$. 
	Say that $Q$ has the {\em \erh{} property} if there exists $c>0$ such that every $n$-vertex $Q$-free tournament admits a transitive subtournament with at least $n^c$ vertices.
	(A warning:  earlier we were using this name applied to hereditary class, and now we are using it for a tournament, not for the class that does not contain it.)
	Alon, Pach, and Solymosi~\cite{MR1832443} proved that \cref{conj:eh} is equivalent to the following statement.
	\begin{conjecture}
		\label{conj:toureh}
		Every tournament has the \erh{} property.
	\end{conjecture}
	For a tournament $Q$ and an ordering $\phi=(v_1,\ldots,v_n)$ of $V(Q)$, the {\em backedge graph} of $Q$ with respect to $\phi$ 
	is the graph $G$ with vertex set $V(Q)$, where for all $i,j\in[n]$ with $i<j$, $v_iv_j\in E(G)$ if and only if $(v_j,v_i)\in E(Q)$.
	For two tournaments $Q_1,Q_2$, the tournament $Q$ obtained by {\em substituting} $Q_2$ for $v\in V(Q_1)$ has vertex set $V(Q_2)\cup (V(Q_1)\setminus\{v\})$, such that $Q[V(Q_1)\setminus\{v\}]=Q_1\setminus v$, $Q[V(Q_2)]=Q_2$,
	and for all $u\in V(Q_1)\setminus \{v\}$ and $w\in V(Q_2)$, $(u,w)\in E(Q)$ if and only if $(u,v)\in E(Q_1)$.
	A tournament is {\em prime} if it cannot be obtained by substitution from tournaments with fewer vertices.
	By an adaptation of another argument of Alon, Pach, and Solymosi in~\cite{MR1832443},
	\cref{conj:eh} reduces to proving that all prime tournaments have the \erh{} property.
	Partial results in this direction include~\cite
	{MR3323031,MR3354292,MR4587738,MR4651833,MR4529840,density4}; these show the \erh{} property of several types of 
prime tournaments, that all admit a forest backedge graph except those from~\cite{density4}. 
	In another direction, Berger, Choromanski, and Chudnovsky~\cite{MR3862957} proved \cref{conj:toureh} for every six-vertex tournament except for the seven-vertex Paley tournament with one vertex removed.

	For an integer $k\ge0$, a tournament is {\em $k$-colourable} if its vertex set can be partitioned into $k$ subsets each inducing a transitive subtournament.
	It is known (see~\cite[Lemma 3.2]{MR3772736}) that 
	a class of tournaments closed under taking subtournaments does
	not contain all two-colourable tournaments if and only if for some $d$,
	the VC-dimension of the set of in-neighbourhoods of each of its members is at most $d$.
	Moreover, a standard argument together with \cref{lem:vcbip} proves the following (see~\cite[Section 4.6]{2025thes} for a proof).
	\begin{lemma}
		\label{lem:vctour}
		For every $d\ge1$, there exists $m\ge1$ such that for every tournament $T$, if the set of its in-neighbourhoods has VC-dimension at most $d$ then every backedge graph of $T$ has VC-dimension at most $m$.
		Conversely, for every $m\ge1$, there exists $d\ge1$ such that
		if a backedge graph of a tournament $T$ has VC-dimension at most $m$,
		then the set of in-neighbourhoods of $T$ has VC-dimension at most $d$.
	\end{lemma}
	Therefore, since every backedge graph of every $n$-vertex transitive tournament is a comparability graph (and so has a clique or stable set of size at least $\sqrt n$), \cref{thm:vc} can be restated as follows.
	\begin{theorem}
		\label{thm:tour}
		Every two-colourable tournament has the \erh{} property.
	\end{theorem}
Tournaments that admit forest backedge graphs are two-colourable, so this
	result contains all of the aforementioned results on prime tournaments with the \erh{} property, except those from~\cite{density4};
	and since all six-vertex tournaments are two-colourable (see~\cite{MR1310883}), \cref{thm:tour} implies that they all satisfy \cref{conj:toureh} as well. 
	One can also formulate a viral version for tournaments similar to the one for graphs and prove that every two-colourable tournament is viral; we omit the details.
	
	\section*{Acknowledgements}
	We would like to thank Artem Chernikov and Jacob Fox for helpful discussions.
	We are also grateful to the anonymous referees for helpful remarks.
	
	\bibliographystyle{abbrv}

\end{document}